\documentclass[11pt]{smfart}
\usepackage{color}
\usepackage{amssymb,verbatim}
\usepackage{hyperref}
\usepackage{amsthm,array,amssymb,amscd,amsfonts,amssymb,latexsym, url}
\usepackage{amsmath}
\usepackage[all]{xy}
\usepackage[french]{babel}

\newcounter{spec}
{\end{list}}

\renewcommand{\P}{{\mathbf P}}

\newcommand{\Betti}{{\rm Betti}}
\newcommand{\A}{{\mathbb A}}

\newcommand{\Z}{{\mathbb Z}}
\newcommand{\Q}{{\mathbb Q}}
\newcommand{\C}{{\mathbb C}}

\renewcommand{\L}{{\mathbb L}}

\newcommand{\Br}{{\operatorname{Br   }}}

\newcommand{\Hom}{{\operatorname{Hom}}}

\newcommand{\Spec}{{\operatorname{Spec \ }}}

\renewcommand{\lim}{\varprojlim}

\numberwithin{equation}{section}

\newfont{\gothic}{eufb10}

\renewcommand{\qed}{{\hfill$\square$}}

\newtheorem{theo}{Th\'{e}or\`{e}me}[section]
\newtheorem{prop}[theo]{Proposition}

\newtheorem{lem}[theo]{Lemme}
\newtheorem{cor}[theo]{Corollaire}
\theoremstyle{definition}
\newtheorem{defi}[theo]{D\'efinition}
\theoremstyle{remark}
\newtheorem{rema}[theo]{Remarque}

\newcommand{\bthe}{\begin{theo}}
\newcommand{\ble}{\begin{lem}}
\newcommand{\bpr}{\begin{prop}}
\newcommand{\bco}{\begin{cor}}
\newcommand{\bde}{\begin{defi}}
\newcommand{\ethe}{\end{theo}}
\newcommand{\ele}{\end{lem}}
\newcommand{\epr}{\end{prop}}
\newcommand{\eco}{\end{cor}}
\newcommand{\ede}{\end{defi}}

\newcommand{\G}{{\mathbb G}}

\newcommand\NS{{\operatorname{NS}}}

\DeclareFontFamily{U}{wncy}{}
\DeclareFontShape{U}{wncy}{m}{n}{%
<5>wncyr5%
<6>wncyr6%
<7>wncyr7%
<8>wncyr8%
<9>wncyr9%
<10>wncyr10%
<11>wncyr10%
<12>wncyr6%
<14>wncyr7%
<17>wncyr8%
<20>wncyr10%
<25>wncyr10}{}
\DeclareMathAlphabet{\cyr}{U}{wncy}{m}{n}

\begin{document}

  \title[Produit  avec une courbe elliptique]{Cohomologie non ramifi\'ee  dans le produit  avec une courbe elliptique}

\author{J.-L. Colliot-Th\'el\`ene}
\address{Laboratoire de Math\'ematiques d'Orsay, B\^{a}timent 309,  Univ. Paris-Sud, CNRS, Universit\'e Paris-Saclay, 91405 Orsay, France }
\email{jlct@math.u-psud.fr}

\date{12 f\'evrier 2018; version r\'evis\'ee, 19 septembre 2018}
\maketitle

 \begin{abstract}
 Un th\'eor\`eme de Gabber (2002) permet de construire des classes de
cohomologie non ramifi\'ee dans le produit de certaines vari\'et\'es
et d'une courbe elliptique.
Le lien  entre la cohomologie non ramifi\'ee en degr\'e 3 et la conjecture de
Hodge enti\`ere pour les cycles de codimension deux  (2012) 
permet alors de donner de nombreuses classes de vari\'et\'es
pour lesquelles la conjecture de Hodge enti\`ere  pour les cycles de
codimension deux est en d\'efaut. Le cas particulier du produit avec
une surface d'Enriques a \'et\'e
\'etabli par Benoist et Ottem (2018). 
  \end{abstract}
  
  \begin{altabstract}
  A method of Gabber (2002) produces unramified cohomology classes in
the products of certain varieties with  an elliptic curve. The connection between third
  unramified cohomology and integral  Hodge conjecture for codimension 2 cycles (2012)
  then gives many examples of such a product for which this conjecture fails.
  The special case of the product with an Enriques surface was established by Benoist and Ottem (2018).

   \end{altabstract}

\section*{}
Sauf mention expresse du contraire, la cohomologie employ\'ee ici est la cohomologie \'etale de SGA4.
On utilise librement les propri\'et\'es de cette derni\`ere, comme on peut les trouver  
 dans le livre \cite{M}. Pour la cohomologie non ramifi\'ee, on
  renvoie le lecteur \`a  \cite{CTV} et aux r\'ef\'erences de cet article.  
  
  \section{Cohomologie non ramifi\'ee en tout degr\'e}

\begin{theo}\label{general}
Soit $X/\C$ une vari\'et\'e connexe,  projective et lisse et $\C(X)$ son corps des fonctions.
Soit $\ell$ un nombre premier.
 Soit $\alpha \in H^{i}(X,\Z/\ell) $ une classe de cohomologie dont
l'image dans $H^{i}(\C(X), \Z/{\ell})$ est non nulle.

(a)  Il existe une courbe elliptique  $E/\C$ et    $\beta \in H^1(E,\mu_{\ell})$
tels que l'image de $\alpha \cup \beta \in H^{i+1}(X \times E,\mu_{\ell} )$
dans $H^{i+1}(\C(X\times E),\mu_{\ell})$ soit non nulle. En particulier,
le groupe de cohomologie non ramifi\'ee $H^{i+1}_{nr}(\C(X\times E),\mu_{\ell})$ est non nul.

(b) Si $X$ peut \^etre d\'efinie sur un corps de nombres, pour toute courbe elliptique $E/\C$
d'invariant $j$ transcendant, le groupe de cohomologie non ramifi\'ee $H^{i+1}_{nr}(\C(X\times E),\mu_{\ell})$ est non nul.
\end{theo}
\begin{proof}
D'apr\`es  Gabber \cite[Prop. A.4]{G}, dont on garde les notations,
il existe une courbe lisse g\'eom\'etriquement connexe $U/\Q$, un point $P\in U(\Q)$, et une 
suite exacte de $U$-sch\'emas en groupes commutatifs lisses connexes
$$ 1 \to  \mu_{\ell,U} \to  {\mathcal E} \to {\mathcal E}' \to 1$$
dont la fibre au-dessus de $P$ s'identifie \`a la suite exacte de Kummer
$$1 \to \mu_{\ell,\Q} \to \G_{m,\Q}  \to \G_{m,\Q} \to 1$$
associ\'ee \`a $x \mapsto x^{\ell}$,
et dont la restriction au-dessus de $V=U\setminus P$ est une isog\'enie de
$V$-sch\'emas ab\'eliens de dimension relative 1. En outre toute $\ell$-isog\'enie
de courbes elliptiques sur $\C$  munie d'un isomorphisme de son noyau
avec $\mu_{\ell}$ est donn\'ee par l'\'evaluation de la suite exacte ci-dessus
en un point de $U(\C)$. L'invariant $j$  des fibres de ${\mathcal E}' \to U$
hors du point $P$
n'est en particulier pas constant, et prend une valeur transcendante sur $\Q$
en $M \in U(\C)$ si et seulement si $M$ n'est pas alg\'ebrique sur $\Q$.
Notons pour simplifier $Y= {\mathcal E}' $.
La suite exacte ci-dessus d\'efinit un torseur sur $Y$ sous $\mu_{\ell}$, donc
une classe $\beta \in H^1(Y,\mu_{\ell})$. La restriction de cette classe
au-dessus du point g\'en\'erique de $Y$ est la classe d'une fonction
rationnelle $g \in \Q(Y)^*/\Q(Y)^{*\ell}= H^1(\Q(Y),\mu_{\ell} )$.
L'extension $\Q(Y)(g^{1/\ell})/\Q(Y)$ se sp\'ecialise au-dessus du
point $P$ en l'extension $\Q(t^{1/\ell})/\Q(t)$, o\`u  
$\G_{m,\Q} = \Spec \Q[t,1/t]$.

Commen\c cons par \'etendre la situation ci-dessus de $\Q$ \`a $\C$.
On note $Y_{\C}=Y\times_{\Q}\C$ et $U_{\C}=U\times_{\Q}\C$. On a la projection $X \times_{\C}Y_{\C} \to U_{\C}$.
On consid\`ere le produit externe  $\alpha \cup \beta \in H^{i+1}(X \times_{\C}Y_{\C}, \mu_{\ell})$.
Il se sp\'ecialise au-dessus du point $P \in U(\C)$ en une classe
dans $H^{i+1}(X \times_{\C} \G_{m,\C}, \mu_{\ell})$.
La restriction de cette classe au point g\'en\'erique de $X \times_{\C}\G_{m,\C}$
est non nulle, car son r\'esidu le long du diviseur d\'efini par $t=0$ dans $X \times \A^1$ est la classe de $\alpha$ dans $H^{i}(\C(X),\Z/{\ell})$.
D'apr\`es  Gabber \cite[Prop.  A7]{G},   l'ensemble des points  $s$ de $U(\C)$ tels que la restriction
de  $\alpha \cup \beta$  au point g\'en\'erique de la fibre g\'eom\'etrique de $X \times_{\C} Y_{\C} \to U_{\C}$ en $s$
est  nulle est une union d\'enombrable de ferm\'es   de $U_{\C}$.
On a vu ci-dessus que $P$ n'est pas dans cet ensemble.
Cet ensemble est donc une union d\'enombrable de points de $U(\C)$. 
Ceci \'etablit le point (a).

 Supposons maintenant $X=X_{0}\times_{k} \C$, o\`u $k\subset \C$ d\'esigne la cl\^oture alg\'ebrique de $\Q$
dans $\C$. La cohomologie \'etale d'une vari\'et\'e   sur un corps alg\'ebriquement clos, 
 \`a coefficients de torsion premiers \`a la caract\'eristique,
ne change pas par extension de corps de base alg\'ebriquement clos \cite[VI.4.3]{M}.
 On dispose donc d'une classe 
 $\alpha_{0}   \in H^{i}(X_{0},\Z/\ell)$ d'image non nulle dans le groupe $H^{i}(k(X_{0}),\Z/\ell)$.
D'apr\`es  Gabber \cite[Prop. A7]{G},   l'ensemble des points  $s \in U_{k} $ tels que la restriction
de  $\alpha_{0} \cup \beta_{k}$  au point g\'en\'erique de la fibre g\'eom\'etrique de $X_{0} \times_{k} Y_{k} \to U_{k}$ en $s$
est  nulle est une union d\'enombrable de ferm\'es   de $U_{k}$, et donc, puisque
$P$ n'est pas dans cet ensemble, une union d\'enombrable de points  de $U(k)$. Pour tout point $M$
de $U(\C) \setminus U(k)$, la restriction de   $\alpha_{0} \cup \beta_{k}$  dans $H^{i}(\C(X\times_{\C}Y_{M}), \Z/\ell)$
est donc non nulle, ce qui \'etablit (b).
\end{proof}

\section{Cohomologie non ramifi\'ee en degr\'e 3 et conjecture de Hodge enti\`ere}

\begin{prop}\label{brauer}
Soit $X/\C$ une vari\'et\'e connexe,  projective et lisse, telle que $\Br(X)\neq 0$.

(a) Il existe une courbe elliptique $E/\C$
telle que $H^{3}_{nr}(\C(X\times E),\Q/\Z) \neq 0$.

(b) Si $X$ peut \^etre d\'efinie sur un corps de nombres, pour toute courbe elliptique $E/\C$
d'invariant $j$ transcendant, $H^{3}_{nr}(\C(X\times E),\Q/\Z) \neq 0$.

(c)  Si   le groupe de Chow des z\'ero-cycles de la vari\'et\'e complexe   $X$  est support\'e sur une courbe, alors
il existe une courbe elliptique $E/\C$
telle que  la conjecture de Hodge enti\`ere pour les cycles de codimension 2 sur $X \times E$ soit en d\'efaut.

(d) Si   le groupe de Chow des z\'ero-cycles de la vari\'et\'e complexe  $X$ est support\'e sur une courbe,
et  $X$ peut \^etre d\'efinie sur un corps de nombres, alors pour toute courbe elliptique $E/\C$
d'invariant $j$ transcendant, la conjecture de Hodge enti\`ere pour les cycles de codimension 2 sur $X \times E$ est en d\'efaut.
\end{prop}
\begin{proof}
Supposons $\Br(X)\neq 0$. Soit $\ell$ un nombre premier avec
 $\Br(X)({\ell})\neq 0$. De la suite de Kummer pour la cohomologie \'etale sur $X$
 et sur le corps des fonctions $\C(X)$, et  de l'injectivit\'e bien connue $\Br(X) \hookrightarrow \Br(\C(X))$
\cite[II, Cor. 1.10]{Gr},
  on conclut qu'il existe $\alpha \in H^2(X,\mu_{\ell})$ d'image non nulle dans
 $H^2(\C(X),\mu_{\ell}) = \Br(\C(X))[\ell]$. Le th\'eor\`eme \ref{general} donne alors
 les points (a) et (b). 
 
 Dans un travail avec C. Voisin, en utilisant des r\'esultats profonds de $K$-th\'eorie alg\'ebrique,
 on a \'etabli le r\'esultat suivant \cite[Thm. 1.1, Thm. 3.9]{CTV} :
 pour toute vari\'et\'e $W$ projective et lisse sur $\C$
 dont le groupe de Chow des z\'ero-cycles $CH_{0}(W)$ est support\'e sur une surface,
 ce qui signifie qu'il existe une surface projective lisse $U$ sur $\C$ et un morphisme $U \to W$ tels que l'application
 induite $CH_{0}(U) \to CH_{0}(W)$ soit surjective,
 on a un isomorphisme de groupes {\it finis}
 $$ H^3_{nr}(W, \Q/\Z) \simeq Z^4(W),$$
 o\`u $ H^3_{nr}(W, \Q/\Z)$ est la r\'eunion des groupes $H^3_{nr}(W,\Z/n)$
 et o\`u $Z^4(W)$ est le 
 quotient du groupe des cycles de Hodge  
 dans $H^4(W,\Z)$ par l'image des classes de cycles alg\'ebriques de codimension 2.
 Si le groupe de Chow des z\'ero-cycles de $X$ est support\'e sur une courbe $C$, alors
 le groupe de Chow des z\'ero-cycles de $X \times_{\C} E$ est support\'e sur la surface
 $C \times_{\C}E$.
 On obtient ainsi (c) et (d).
\end{proof}

Outre les surfaces d'Enriques, de nombreuses vari\'et\'es dont le groupe
de Chow des z\'ero-cycles est support\'e sur une courbe satisfont 
$\Br(X) \neq 0$.
 
Rappelons ici  que pour toute vari\'et\'e $X/\C$ projective, lisse, connexe,
on dispose d'une suite exacte
$$ 0 \to (\Q/\Z)^{b_{2}-\rho} \to \Br(X) \to H^3_{\Betti}(X,\Z)_{tors} \to 0,$$
o\`u $\rho$ est le rang du groupe de N\'eron-Severi $\NS(X)$ et $b_{2}$ le second nombre de Betti
de $X$
(voir \cite[III. (8.7) et (8.9)]{Gr}), et que l'on a $b_{2}-\rho=0$ si et seulement si $H^2(X,O_{X})=0$. 
Bloch  \cite{Bl}  a montr\'e que l'hypoth\`ese que
 le groupe de Chow des z\'ero-cycles sur $X$ est support\'e sur une courbe
 implique $H^2(X,O_{X})=0$ et  que  le groupe $ \Br(X)$ est fini, et \'egal au groupe $H^3_{\Betti}(X,\Z)_{tors}$.
 
Pour $X/\C$ une surface projective, lisse, connexe,
de groupe de N\'eron-Severi $\NS(X)$,
notant $\Br(X)^0$ le sous-groupe divisible maximal de $\Br(X)$,
on a un isomorphisme naturel \cite[III (8.12)]{Gr} :
$$\Br(X)/\Br(X)^0 \simeq \Hom(\NS(X)_{tors}, \Q/\Z).$$
Dans le cas particulier o\`u $H^2(X,O_{X})=0$, en particulier
si le groupe de Chow des z\'ero-cycles de la surface $X$ est support\'e sur une courbe,
on a donc 
$$\Br(X) \simeq \Hom(\NS(X)_{tors}, \Q/\Z).$$
Pour les surfaces, c'est  une conjecture c\'el\`ebre
de Spencer Bloch que l'hypoth\`ese $H^2(X,O_{X})=0$
implique que le groupe de Chow des z\'ero-cycles sur $X$ est support\'e sur une courbe.
Cette conjecture est
connue dans le cas des surfaces non de type g\'en\'eral
\cite{BKL},  en particulier pour les surfaces d'Enriques, 
et dans quelques autres cas \cite{voisin}.

Dans le cas plus particulier d'une surface $X$
satisfaisant $q= p_{g}=0$, i.e. avec $H^{i}(X,O_{X})=0$ pour $i=1,2$,
comme par exemple une surface d'Enriques,
on~a  $$\Br(X) \simeq \Hom(\NS(X)_{tors}, \Q/\Z).$$

De nombreuses surfaces avec $p_{g}=q=0$ et
$\NS(X)_{tors} \neq 0$ ont \'et\'e d\'ecrites (surfaces de Godeaux, surfaces de Campedelli
(voir \cite[Chap. VII, \S 11]{BPV} et  \cite{BCGP}).
 
La proposition  \cite{Bl}  et les rappels ci-dessous donnent   :

\begin{cor}\label{enriquesgen}
Soit $X$ une surface sur $\C$ dont le groupe de N\'eron-Severi $\NS(X)$ poss\`ede de la torsion,
et dont le groupe de Chow des z\'ero-cycles est support\'e sur une courbe, donc  qui satisfait
$H^2(X,O_{X})=0$.
  Il existe une courbe elliptique $E/\C$
telle que  la conjecture de Hodge enti\`ere pour les cycles de codimension 2 sur $X \times E$ 
soit en d\'efaut. Si $X$ peut \^etre d\'efinie sur un corps de nombres, cette d\'efaillance a lieu pour toute
courbe elliptique $E$ d'invariant $j$ transcendant. $\Box$
\end{cor}

\begin{rema}
  Dans le cas particulier des surfaces d'Enriques,
   on retrouve donc ainsi  \cite[Prop. 2.1]{BO}, dont
 la d\'emonstration   repose, dans le cas $\ell=2$,
 sur une construction du m\^eme type que celle 
  de \cite{G}  rappel\'ee au d\'ebut du th\'eor\`eme \ref{general}.
La d\'emonstration de \cite[Prop. 2.1]{BO} pour les surfaces d'Enriques
est plus ``classique''
   que celle pr\'esent\'ee ici : elle n'utilise pas de r\'esultat de $K$-th\'eorie alg\'ebrique.
\end{rema}

\begin{rema} D\`es la dimension 3, il existe des vari\'et\'es projectives et lisses $X$ sur $\C$,
rationnellement connexes, telles que $\Br(X) \neq 0$, comme les exemples d'Artin et Mumford
de solides fibr\'es en coniques sur le plan projectif $\P^2_{\C}$. On peut facilement donner de tels
exemples qui sont d\'efinis sur un corps de nombres.
Pour les vari\'et\'es rationnellement connexes, le groupe de Chow des z\'ero-cycles est support\'e sur un point.
Il est donc facile d'exhiber {\it des vari\'et\'es   $W= X \times_{\C} E$  de dimension 4 explicites, produits d'une
vari\'et\'e rationnellement connexe de dimension 3 et d'une courbe elliptique, pour
lesquelles la conjecture de Hodge enti\`ere pour les cycles de codimension deux est en
d\'efaut} : on prend la vari\'et\'e  produit d'un exemple d'Artin-Mumford d\'efini sur un corps de nombres et 
d'une courbe elliptique d'invariant $j$ transcendant. Notons que  Schreieder  \cite[Cor. 1.6]{Sch} a r\'ecemment
d\'emontr\'e   qu'en dimension au moins 4 il existe
des hypersurfaces de Fano (donc rationnellement connexes)  pour lesquelles
la conjecture de Hodge enti\`ere pour les cycles de codimension deux est en d\'efaut. 
\end{rema}

\end{document}